\documentclass[preprint,12pt]{elsarticle}

\usepackage{amssymb}
\usepackage{amsthm}
\usepackage{artex}
\journal{}

\begin{document}
\begin{frontmatter}



\title{GK-Dimension of some Connected Hopf Algebras}


\author{Daniel Yee}

\address{Bradley University, Peoria, IL 61606}

\begin{abstract}While it was identified that the growth of any connected Hopf algebras is either a positive integer or infinite (see \cite{zhu}), we have yet to determine the GK-dimension of a given connected Hopf algebra. We use the notion of anti-cocommutative elements introduced in \cite{wzz3} to analyze the structure of connected Hopf algebras generated by anti-cocommutative elements and compute the Gelfand-Kirillov dimension of said algebras. Additionally, we apply these results to compare global dimension of connected Hopf algebras and the dimension of their corresponding Lie algebras of primitive elements. \end{abstract}


\end{frontmatter}

\section{Introduction}
\par Connected Hopf algebras are generalizations of universal enveloping algebras $U(\fk{g})$ with respect to the Hopf structure, so it is natural to ask if any of the ring-theoretic properties of enveloping algebras hold for connected Hopf algebras. Throughout the article, we focus on the Gelfand-Kirillov dimension, denoted GK.dim, and the global dimension, denoted gl.dim of some connected Hopf algebras.
\par We assume that $k$ is an algebraically closed field $k$ of characteristic zero. All vector spaces, algebras, tensor products, affine-ness and linear maps are over $k$. We let $\tau:V\otimes W\ra W\otimes V$ represent the twist map $\tau(v\otimes w)=w\otimes v$, and let $\tau\circ\delta$ be the composition between two linear maps $\tau$ and $\delta$.
\par Given a Hopf algebra $H$, we denote $\Delta:H\ra H\otimes H$ as comultiplication on $H$ with the Sweedler notation\begin{align*}\Delta(h)=\sum_hh_1\otimes h_2.\end{align*}Additionally, let $S:H\ra H$ be the antipode of $H$, $P(H)$ be the Lie subalgebra of primitive elements of $H$, i.e.\begin{align*}P(H)=\{x\in H:\Delta(x)=x\otimes1+1\otimes x\},\end{align*}let $\{H_n:n\geq0\}$ be the coradical filtration on $H$, and $\fx{gr}H$ be the associated graded algebra with respect to the coradical filtration. Because $P(H)$ is a Lie algebra, there exists a Hopf algebra isomorphism between the enveloping algebra $U(P(H))$ and the Hopf subalgebra of $H$ generated by $P(H)$, which we will denote this Hopf subalgebra as $U_H$.
\par Recall that a Hopf algebra $H$ is \textit{connected} if $H_0=F$, and a connected Hopf algebra is \textit{locally finite} if every coradical filter is finite dimensional, or equivalently the Lie algebra of primitive elements is finite dimensional. Furthermore, given a connected Hopf algebra $H$, for any $n\geq1$ and for every $h\in H_n$, we have $\Delta(h)=h\otimes1+1\otimes h+w$, where $w\in H_{n-1}\otimes H_{n-1}$ (see \cite[Lemma 5.3.2]{mon}), and let $\ol{h}$ be the element in the subspace $H_n/H_{n-1}$.
\par Given a connected Hopf algebra $H$, we take into consideration a particular vector space\begin{align*}P_2(H):=\{c\in H:\tau\circ\delta(c)=-\delta(c),\ \x{and}\ \delta(c)\in P(H)\otimes P(H)\},\end{align*}where the linear map $\delta:H\ra H\otimes H$ is defined by\begin{align*}\delta(h)=\Delta(h)-(h\otimes1+1\otimes h),\end{align*}for all $h\in H$. We call elements belonging to $P_2(H)$ \textit{anti-cocommutative}; it is clear that $\delta(P(H))=0$, so primitive elements are anti-cocommutative. The reader can find more about anti-cocommutativity and $P_2(H)$ in \cite{wzz5}. Finally, we say that $H$ is \textit{primitively thick} if $\x{GK.dim}(H)=\dim P(H)+1<\infty$, or equivalently, if $H$ is generated by $P_2(H)$ as a Hopf algebra, and $\dim P_2(H)=\dim P(H)+1$ is finite (see \cite[Lemma 2.6]{wzz5} and \cite[Theorem 6.2.12]{gil}).
\par In section 1, we determine the GK-dimension of some connected Hopf algebras $H$. In particular, our focus is on Hopf algebras generated by $P_2(H)$, so we will denote $\cl{A}(\fk{g})$ as the class of locally finite connected Hopf algebras $H$ generated by $P_2(H)$ given a finite dimensional Lie algebra $\fk{g}=P(H)$ and $H\neq U_H\cong U(\fk{g})$. It immediately follows that every algebra in $\cl{A}(\fk{g})$ is affine. To calculate GK-dimension of these algebras, we need the normality condition. Recall that given a Hopf algebra $H$, a Hopf subalgebra $A$ of $H$ is \textit{normal} if both\begin{align*}\x{ad}_l[h](a):=\sum_hh_1aS(h_2)\in A,\ \x{and}\ \x{ad}_r[h](a):=\sum_hS(h_1)ah_2,\end{align*}for every $h\in H$ and every $a\in A$.

\begin{thm}\label{00}If $H\in\cl{A}(\fk{g})$, and if $U_H$ is a normal Hopf subalgebra of $H$, then $\x{GK.dim}(H)=\dim P_2(H)$.\end{thm}

\par Theorem \ref{00} is analogous to the fact that $\x{GK.dim}(U(\fk{g}))=\dim\fk{g}$, and $\fk{g}$ is the generating subspace of $U(\fk{g})$. With finite GK-dimension, algebras in $\cl{A}(\fk{g})$ have nice ring-theoretic properties thanks to \cite[Corollary 6.10]{zhu}.

\begin{cor}\label{01}If $H\in\cl{A}(\fk{g})$, and $U_H$ is a normal Hopf subalgebra of $H$, then $H$ is a Noetherian, Cohen-Macaulay, Auslander-regular domain with $\x{gl.dim}(H)=\x{GK.dim}(H)$.\end{cor}

\par In section 2, we apply the tools in proving Theorem \ref{00} into the following result on global dimension of connected Hopf algebras. Note that $\x{gl.dim}(H)=\x{l.gl.dim}(H)=\x{r.gl.dim}(H)$ for any Hopf algebra $H$ due to \cite[Proposition A.1]{wyz}.

\begin{thm}\label{02}Suppose $H $is any locally finite connected Hopf algebra with\begin{align*}\dim P(H)=\x{gl.dim}(H)<\infty.\end{align*}If one of the following conditions holds\begin{enumerate}
\item $P(H)$ is (completely) solvable, or
\item $U_H$ is a normal Hopf subalgebra of $H$,\end{enumerate}then $H=U_H$.\end{thm}

\par The motivation behind Theorem \ref{02} is to mimic \cite[Lemma 7.2]{zhu} but we replace GK-dimension for global dimension.

\section{Gelfand-Kirillov Dimension}
\par To compute GK-dimension, we need to connect normal Hopf algebras with a ring theoretic notion: almost centralizing extensions.
\par Given algebras $R\subset S$, where $S$ is generated by $\{x_1,...,x_d\}$ over $R$, we say that $S$ is an \textit{almost centralizing extension} of $R$ if the following hold:\begin{enumerate}
\item $[r,x_i]:=rx_i-x_ir\in R$ for all $r\in R$,
\item $[x_i,x_j]:=x_ix_j-x_jx_i\in\sum_{m=1}^dx_mR+R$,
\end{enumerate}for all $i,j\leq d$.

\begin{lmm}\label{10}Let $A\subset H$ be connected Hopf algebras, and let $\{h_1,...,h_d\}$ generate $H$ over $A$ such that $\delta(h_i)\in A\otimes A$ for all $i\leq d$. Then $H$ is an almost centralizing extension of $A$ if and only if both conditions are satisfied:\begin{enumerate}
\item $A$ is a normal Hopf subalgebra of $H$, and
\item $\delta([h_i,h_j])\in\sum_{m=1}^d\delta(h_mA)+\delta(A)$ for all $i,j\leq d$.
\end{enumerate}\end{lmm}
\begin{proof}First, assume that $H$ is an almost centralizing extension. It is clear that condition 2 of the definition implies condition 2, since $\delta$ is a linear map. To show normality, consider any $a\in A$. Since the adjoint map $\x{ad}_l[h]$ is linear and $\x{ad}_l[bc]=\x{ad}_l[c]\circ\x{ad}_l[b]$, for all $b,c\in H$, then without loss of generality let $h\in\{h_1,...,h_d\}$ and $\Delta(h)=h\otimes1+1\otimes h+\sum_hh_1\otimes h_2$. It follows that $S(h)=-h-\sum_hh_1S(h_2)$, and \begin{align*}\x{ad}_l[h](a)&=ha+aS(h)+\sum_hh_1aS(h_2)\\
&=ha-ah-\sum_hah_1S(h_2)+\sum_hh_1aS(h_2)\\
&=[h,a]+\sum_h[h_1,a]S(h_2).\end{align*}Our assumption $\delta(h_i)\in A\otimes A$ and $A$ is a Hopf subalgebra implies that $\sum_h[h_1,a]S(h_2)\in A$. Furthermore, almost centralizing extension implies that $[h,a]\in A$, hence $\x{ad}_l[h](A)\subset A$. Similarly, $\x{ad}_r[h](A)\subset A$, therefore $A$ is a normal Hopf subalgebra of $H$.
\par Conversely, assume that both 1 and 2 hold. Normality combined with the adjoint computation from the previous paragraph shows that $[h,a]=\x{ad}_l[h](a)-\sum_h[h_1,a]S(h_2)\in A$, for all $a\in A$ and all $h\in\{h_1,...,h_d\}$.
\par With 2, write $\delta([h_i,h_j])=\sum_{m=1}\delta(h_ma_m)+\delta(a_0)$, where $a_0,...,a_m\in A$. Since we have\begin{align*}\Delta([h_i,h_j])&=[h_i,h_j]\otimes1+1\otimes[h_i,h_j]+\delta([h_i,h_j]),\\
\Delta(h_ma_m)&=h_ma_m\otimes1+1\otimes h_ma_m+\delta(h_ma_m),\\
\Delta(a_0)&=a_0\otimes1+1\otimes a_0+\delta(a_0),\end{align*}for all $m\leq d$, it follows that\begin{align*}\Delta\left([h_i,h_j]-\sum_{m=1}^dh_ma_m-a_0\right)&=\left([h_i,h_j]-\sum_{m=1}^dh_ma_m-a_0\right)\otimes1\\
&\qquad+1\otimes\left([h_i,h_j]-\sum_{m=1}^dh_ma_m-a_0\right),\end{align*}whence $[h_i,h_j]-\sum_{m=1}^dh_ma_m-a_0\in P(H)$. By our assumption, the generating set $\{h_1,...,h_d\}$ contains the basis elements of $P(H)$ that are not in $P(A)$. Thus, we may assume that $[h_i,h_j]-\sum_{m=1}^dh_ma_m-a_0=0$. Since $i,j\leq d$ are arbitrary, then $[h_i,h_j]\in\sum_{m=1}h_mA+A$, hence $H$ is an almost centralizing extension of $A$.\end{proof}

\par Though the next corollary is not exactly Theorem \ref{00}, there are cases where it may be used to compute the GK-dimension of connected Hopf algebras.

\begin{cor}\label{11}Suppose $A\subset H$ are affine locally finite connected Hopf algebras such that $\{h_1,...,h_d\}\subset H$ generate $H$ over $A$. If the following hold:\begin{enumerate}
\item $\delta(h_i)\in A\otimes A$ for all $i\leq d$,
\item $\delta([h_i,h_j])\in\sum_{m=1}^d\delta(h_mA)+\delta(A)$ for all $i,j\leq d$,
\item $A$ is a normal Hopf subalgebra of $H$ with finite GK-dimension,\end{enumerate}then $\x{GK.dim}(H)=\x{GK.dim}(A)+d<\infty$.\end{cor}
\begin{proof}Lemma \ref{10} states that $H$ must be an almost centralizing extension of $A$. By \cite[Proposition 6.4]{zhu} the associated graded algebra with respect to the coradical filtration is $\fx{gr}H=\fx{gr}A[\ol{h}_1,...,\ol{h}_d]$, the commutative polyonomial ring over $\fx{gr}A$. Thus, by \cite[Proposition 8.6.7]{mcr} we have that $\x{GK.dim}(\fx{gr}H)=\x{GK.dim}(\fx{gr}A)+d$. Assuming $A$ has finite GK-dimension implies that $\fx{gr}A$ is affine and $\fx{gr}A$ has finite GK-dimension by \cite[Theorem 6.9]{zhu}, and hence $\fx{gr}H$ has finite GK-dimension. Therefore,\begin{align*}\x{GK.dim}(H)=\x{GK.dim}(\fx{gr}H)=\x{GK.dim}(A)+d<\infty,\end{align*}as desired.\end{proof}

\begin{exa}Let's consider the connected Hopf algebra $H$ generated by $kx_1+kx_2+kx_3+kz$ with the relations\begin{align*}[x_1,x_2]&=x_2\qquad[x_1,x_3]=[x_2,x_3]=0\\
[z,x_1]&=z\qquad[z,x_2]=0\qquad[z,x_3]=x_2,\end{align*}and the Hopf structure is given by\begin{align*}\varepsilon(kx_1+kx_2+kx_3+kz)&=0,\ \text{where $\varepsilon$ is the counit,}\\
\Delta(x_i)&=x_i\otimes1+1\otimes x_i,\ \text{for $i=1,2,3$,}\\
\Delta(z)&=z\otimes1+1\otimes z+x_1\otimes x_2-x_2\otimes x_1.\end{align*}(This algebra $H$ is \cite[Theorem 3.5(a)]{wzz5} with $a=c=0$ and $b=1$.) Clearly every $x_i\in P(H)$ and $z\in P_2(H)$. Because of the relation $[z,x_1]=z$, it follows that $U_H$ is not a normal Hopf subalgebra of $H$, thus we cannot apply Theorem \ref{00}. But we can apply Corollary \ref{11} to compute $\x{GK.dim}(H)$.
\par Let $A$ be the Hopf subalgebra generated by $kx_1+kx_2+kz$. It follows that $A$ is algebra-isomorphic to the enveloping algebra $U(\fk{g})$, where the Lie algebra $\fk{g}=kX_1+kX_2+kZ$ has the relations $[X_1,X_2]=[Z,X_2]=0$, $[Z,X_1]=Z$. Hence $\x{GK.dim}(A)=3$. Obviously $\{x_3\}$ generates $H$ over $A$ and $\delta(x_3)\in A\otimes A$, thus satisfying condition 1. Condition 2 is trivial, and normality immediately follows from $\x{ad}_l[x_3](a)=-\x{ad}_r[x_3](a)=x_3a-ax_3\in A$ for any $a\in kx_1+kx_2+kz$. Therefore, Corollary \ref{11} implies that $\x{GK.dim}(H)=\x{GK.dim}(A)+1=4$.\end{exa}

\par We focus our attention to proving Theorem \ref{00}. We first decompose the linear map $\delta:H\ra H\otimes H$ to two distinct parts.

\begin{defn}Given a connected Hopf algebra $H$, we define\begin{align*}\delta_{cc}=\tfrac{1}{2}(\delta+\tau\circ\delta),\ \x{and}\ \delta_{ac}=\tfrac{1}{2}(\delta-\tau\circ\delta)\end{align*}Note that $\delta=\delta_{cc}+\delta_{ac}$.\end{defn}

\begin{lmm}\label{12}Suppose $H$ is a connected Hopf algebra with $P=P_2(H)$ and $U=U_H$. Then\begin{enumerate}
\item $\delta_{cc}([s,t])=[\delta(s),\delta(t)]$ in $H\otimes H$, for any $s,t\in P_2(H)$.
\item $\delta_{ac}|_U=0$ while $\delta_{cc}|_U=\delta|_U$.
\item $\delta_{ac}|_P=\delta|_P$ while $\delta_{cc}|_P=0$.
\end{enumerate}\end{lmm}
\begin{proof}In $H\otimes H$, notice that for any non-primitive $s,t\in P$, we have\begin{align*}\delta([s,t])=[(s\otimes1+1\otimes s),\delta(t)]+[\delta(s),(t\otimes1+1\otimes t)]+[\delta(s),\delta(t)].\end{align*}Applying the twist map yields\begin{align*}\tau\circ\delta([s,t])=-[(s\otimes1+1\otimes s),\delta(t)]-[\delta(s),(t\otimes1+1\otimes t)]+[\delta(s),\delta(t)].\end{align*}Therefore, $(\delta+\tau\circ\delta)([s,t])=2[\delta(s),\delta(t)]$, and $(\delta-\tau\circ\delta)([s,t])=2[(s\otimes1+1\otimes s),\delta(t)]+2[\delta(s),(t\otimes1+1\otimes t)]$.
\par The rest is straightforward.\end{proof}

\par In short, $\delta_{cc}$ preserves the cocommutative part of $\delta$, while $\delta_{ac}$  preserves the anti-cocommutative part.
\par Given a connected Hopf algebra $H$, since $P_2(H)$ is the largest subcoalgebra of $H$ consisting of anti-cocommutative elements (\cite[Lemma 2.5]{wzz5}), then one would expect that the preimage of $\delta_{ac}$  belongs to $P_2(A)$, and similarly the preimage of $\delta_{cc}$ belongs to $U(P(H))$.

\begin{lmm}\label{14}Suppose $H\in\cl{A}(\fk{g})$. Let $s,t\in P_2(H)$ be non-primitive, and let $U_n$ denote the coradical filtration of $U_H$, i.e. $U_n=U_H\cap H_n$ for all $n\geq0$. Then $\delta_{cc}([s,t])\in\delta(U_3)$ if and only if $\delta_{ac}([s,t])\in\delta(P_2(H))$.\end{lmm}
\begin{proof}Assume $\delta_{cc}([s,t])=\delta(w)$ for some $w\in U_3$. By Lemma \ref{12}, we have\begin{align*}\Delta([s,t]-w)=([s,t]-w)\otimes1+1\otimes([s,t]-w)+\delta_{ac}([s,t]),\end{align*}thus $\delta([s,t]-w)=\delta_{ac}([s,t])$. Since $\tau\circ\delta_{ac}=-\delta_{ac}$, then $[s,t]-w\in P_2(H)$, whence $\delta([s,t]-w)=\delta_{ac}([s,t])\in\delta(P_2(H))$.
\par Now let $\delta_{ac}([s,t])=\delta(v)$ for some $v\in P_2(A)$. By Lemma \ref{12}, we have\begin{align*}\Delta([s,t]-v)=([s,t]-v)\otimes1+1\otimes([s,t]-v)+\delta_{cc}([s,t]),\end{align*}which implies that $[s,t]-v$ is cocommutative. Since $U_H$ is the largest cocommutative subcoalgebra of $H$, then $[s,t]-v\in U_H$. It follows that $st\in H_4$, whence $[s,t]-v\in H_4\cap U_H=U_4$. Since $\delta_{cc}(v)=0$, we have $\delta_{cc}([s,t])=\delta([s,t]-v)\in\delta(U_4)$. To show that $\delta_{cc}([s,t])\in\delta(U_3)$, consider for any $x_1,x_2,y_1,y_2\in P(H)$,\begin{align*}&[(x_1\otimes y_1-y_1\otimes x_1),(x_2\otimes y_2-y_2\otimes x_2)]\\
&=x_1x_2\otimes y_1y_2-x_1y_2\otimes y_1x_2-y_1x_2\otimes x_1y_2+y_1y_2\otimes x_1x_2-x_2x_1\otimes y_2y_1\\
&\qquad+y_2y_1\otimes x_2x_1+x_2x_1\otimes y_2y_1-y_2y_1\otimes x_2x_1\\
&=x_2x_1\otimes[y_1,y_2]+[y_1,y_2]\otimes x_2x_1+y_2y_1\otimes[x_1,x_2]+[x_1,x_2]\otimes y_2y_1+[x_1,x_2]\otimes[y_1,y_2]\\
&\;+[y_1,y_2]\otimes[x_1,x_2]-(x_1x_2\otimes[y_1,x_2]+[y_1,x_2]\otimes x_1y_2+y_1x_2\otimes[x_1,y_2]+[x_1,y_2]\otimes y_1x_2)\\
&\;\;\; +[x_1,y_2]\otimes[y_1,x_2]+[y_1,x_2]\otimes[x_1,y_2]\\
&\in(U_2/U_1)\otimes U_1+U_1\otimes(U_2/U_1).\end{align*}Because $\delta(s),\delta(t)\in P(H)\otimes P(H)$, then $\delta_{cc}([s,t])\in(U_2/U_1)\otimes U_1+U_1\otimes(U_2/U_1)$. If $[s,t]-v\notin U_3$, then $\delta_{cc}([s,t])=v_1+v_2+u$, where $v_1\in(U_3/U_2)\otimes U_1$, $v_2\in U_1\otimes(U_3/U_2)$ are both nonzero, and $u\in U_2\otimes U_2$. But this is absurd, therefore $[s,t]-v\in U_3$, hence $\delta_{cc}([s,t])\in\delta(U_3)$.\end{proof}

\par We need an equivalent condition of normality for the Hopf subalgebra generated by primitive elements, namely one that is analogous to the statement: given a Lie algebra $\fk{g}$, a subspace $\fk{j}$ of $\fk{g}$ is an ideal if and only if $U(\fk{j})$ is a normal Hopf subalgebra of $U(\fk{g})$ (see \cite{mmo}).

\begin{lmm}\label{15}Suppose $H\in\cl{A}(\fk{g})$. Then $U_H$ is a normal Hopf subalgebra of $H$ if and only if $[t,\fk{g}]\subset\fk{g}$ for all $t\in P_2(H)$.\end{lmm}
\begin{proof}Assume $t\in P_2(H)$ is non-primitive, as $t\in\fk{g}$ is trivial. Since $t$ is a linear combination of the basis of $P_2(H)$, then without loss of generality, assume that $t\in P_2(H)$ is a basis element with $\delta(t)=x\otimes y-y\otimes x$. It follows that $S(t)=-t+[x,y]$, so for any $g\in\fk{g}$,\begin{align*}\x{ad}_r[t](g)&=S(t)g+gt-xgy+ygx\\
&=-tg+xyg-yxg+gt+ygx-xgy\\
&=-[t,g]+y[g,x]+x[y,g],\\
\x{ad}_l[t](g)&=[t,g]+[g,x]y+[y,g]x.\end{align*}If we assume that $U_H$ is a normal Hopf subalgebra, then $[t,g]\in U_H$, and since $[t,g]\in P_2(H)$, we have that $[t,g]\in U_H\cap P_2(H)=P_2(U_H)=\fk{g}$.
\par Conversely, given $[t,\fk{g}]\subset\fk{g}$, we have $\x{ad}_r[t](\fk{g})\subset U_H$, for all $t\in P_2(H)$. Since $\x{ad}_r[ba]=\x{ad}_r[a]\circ\x{ad}_r[b]$ for all $a,b\in H$, then it follows that $\x{ad}_r[H](U_H)\subset U_H$, and similarly $\x{ad}_l[H](U_H)\subset U_H$, therefore $U_H$ is a normal Hopf subalgebra of $H$.\end{proof}

\par We can now prove Theorem \ref{00}.

\begin{proof}[Proof of Theorem \ref{00}.] Since $P_2(H)$ is finite dimensional by \cite[Lemma 2.5]{wzz5}, write $P_2(H)=(\sum_{i=1}^nkt_i)\oplus P(H)$, where each $t_i$ is non-primitive. By definition, $H$ is generated by $\{t_1,...,t_n\}$ over $U_H$. Because $\delta(P_2(H))\subset P(H)\otimes P(H)$, showing condition 2 of Corollary \ref{11} proves our result.
\par Without loss of generality, assume that $\delta(t_1)=x_1\otimes x_2-x_2\otimes x_1$, and $\delta(t_2)=y_1\otimes y_2-y_2\otimes y_1$, where $x_1,x_2,y_1,y_2\in P(H)$. Then in $H\otimes H$ we have\begin{align*}\delta_{ac}([t_1,t_2])&=\delta([t_1,t_2])-[\delta(t_1),\delta(t_2)]\\
&=([t_1,y_1]\otimes y_2-y_2\otimes[t_1,y_1])+(y_1\otimes[t_1,y_2]-[t_1,y_2]\otimes y_1)\\
&\qquad+([x_1,t_2]\otimes x_2-x_2\otimes[x_1,t_2])+(x_1\otimes[x_2,t_2]-[x_2,t_2]\otimes x_1).
\end{align*}Lemma \ref{15} implies that $[t_i,P(H)]\subset P(H)$ for all $i\leq n$, whence $\delta_{ac}([t_1,t_2])\in\delta(P_2(H))$. Furthermore, by Lemma \ref{14}, $\delta_{cc}([t_1,t_2])\in\delta(U_H)$. Therefore,\begin{align*}\delta([t_1,t_2])&=\delta_{ac}([t_1,t_2])+\delta_{cc}([t_1,t_2])\\
&\in\delta(P_2(H))+\delta(U_H)\\
&\qquad\subset\sum_{i=1}^n\delta(t_iU_H)+\delta(U_H).\end{align*}Since $t_1,t_2\in\{t_1,...,t_n\}$ are arbitrary, then condition 2 holds, and hence $H$ is an almost centralizing extension of $U_H$. Therefore,\begin{align*}\x{GK.dim}(H)=\dim P(H)+n=\dim P_2(H),\end{align*}by Corollary \ref{11}.\end{proof}

\begin{exa}There are many examples that satisfy Theorem \ref{00}. For example, the connected Hopfa algebras generated by the anti-cocommutative coassociative Lie algebra $L$ in \cite[Theorem 3.5(b), (c), (d), (g)]{wzz5}, the connected Hopf algebra $A$ in \cite[Example 4.1]{wzz5}, and the construction of the connected Hopf algebra $L$ in \cite[Section 7.2.3]{gil}.\end{exa}

\par Unfortunately not all $H\in\cl{A}(\fk{g})$ satisfies Theorem \ref{00}, especially when $\fk{g}$ is a semisimple Lie algebra.

\begin{thm}\label{18}If $H$ is a locally finite connected Hopf algebra such that $P(H)$ is a semsimple Lie algebra, and if $U_H$ is a normal Hopf subalgebra, then we have $H=U_H$.\end{thm}
\begin{proof}If $H\neq U_H$, then by \cite[Lemma 2.6]{wzz5}, $P_2(H)\neq P(H)$. Since $U_H$ is a normal Hopf subalgebra of $H$, for any non-primitive $t\in P_2(H)$, $[P(H),t]\subset P(H)$ by Lemma \ref{12}, hence if $A$ is the subalgebra of $H$ generated by the subspace $P(H)\oplus kt$, then $A$ is a Hopf sublagebra since $\delta(t)\in P(H)\otimes P(H)$. Due to the relation, if we view the subspace $\fk{h}=P(H)\oplus kt$ as a Lie algebra, it follows that $A$ is isomorphic to the enveloping algebra $U(\fk{h})$ as algebras, hence $\x{GK.dim}(A)=\dim P(H)+1$, hence $A$ is primitively thick. But this contradicts \cite[Proposition 6.4.5]{gil}, therefore we must have $H=U_H$.\end{proof}

\par Transitioning into our next section on global dimension, we first pose the following question.

\begin{que}If $H$ is a locally finite connected Hopf algebra, is the GK-dimension of $H$ finite if and only if the global dimension of $H$ is finite? And does $\x{GK.dim}(H)=\x{gl.dim}(H)$ when either is finite?\end{que}

\par The result \cite[Corollary 6.10]{zhu} showed that whenever a connected Hopf algebra $H$ has finite GK-dimension, then the global dimension $H$ is exactly the GK-dimension of $H$, which proves only one direction.

\section{Application: Global Dimension}
\par In this section we ask if the dimension of $P(H)$ is exactly the global dimension of the connected Hopf algebra $H$, is $H$ isomorphic to the enveloping algebra $U(P(H))$?

\begin{lmm}\label{20}Suppose that $H$ is any connected Hopf algebra and $A$ is a Hopf subalgebra of $H$. Then $\x{gl.dim}(A)\leq\x{gl.dim}(H)$ when $A$ is left Noetherian with finite left global dimension.\end{lmm}
\begin{proof}By \cite{tak}, $H$ is a faithfully flat left $A$-module. As $A$ is left Noetherian with finite global dimension, \cite[Theorem 7.2.6]{mcr} implies that $\x{gl.dim}(A)\leq\x{gl.dim}(H)$.\end{proof}

\begin{cor}\label{21}If $H$ is a locally finite connected Hopf algebra, then $\dim P(H)\leq\x{gl.dim}(H)$.\end{cor}
\begin{proof}Apply Lemma \ref{20} with $A=U_H$.\end{proof}

\par Before proving Theorem \ref{02}, we need to see how the Lie algebra $P(H)$ effects the structure of $H$.

\begin{lmm}\label{24}Suppose $H\in\cl{A}(\fk{g})$ where $\fk{g}$ is a (completely) solvable Lie algebra. Then there exists a primitively thick Hopf subalgebra $A$ of $H$ such that $U_H\subset A$.\end{lmm}
\begin{proof}Clearly the vector space $P_2(H)/\fk{g}$ is a $\fk{g}$-module. Applying \cite[Corollary 2.4.3]{graf}, there exists $v\in P_2(H)/\fk{g}$ such that $x(v)=\lambda(x)v$ for all $x\in\fk{g}$, where $\lambda:\fk{g}\ra k$ is a linear map. Since $x(v)=[x,v]$ in $H$, we let $A$ be the subalgebra of $H$ generated by $\fk{g}\oplus kv$. It is clear that $U_H\subset A$, and $A$ is a Hopf algebra since $\Delta(v)\in\fk{g}\otimes\fk{g}$. Additonally, we may view $\fk{h}=\fk{g}\oplus kv$ as a Lie algebra since $[\fk{g},v]=kv$. It follows that $A$ is isomorphic to the enveloping algebra $U(\fk{h})$ as algebras, and thus $\x{GK.dim}(A)=\dim\fk{g}+1$, hence $A$ is primitively thick.\end{proof}

\par We can now prove Theorem \ref{02}.

\begin{proof}[Proof of Theorem \ref{02}.] We start by assuming the contrary, $H\neq U_H$. As a consequence, $P_2(H)\neq P(H)$ by \cite[Lemma 2.6]{wzz5}.
\par1. First we assume that $P(H)$ is (completely) solvable. By Lemma \ref{24}, there exists a primitively thick Hopf subalgebra $A$ of $H$ such that $A$ is algebra-isomorphic to an enveloping algebra $U(\fk{h})$, where $ \fk{h}$ is a Lie algebra of dimension $\dim P(H)+1$. This implies that $A$ is Noetherian and $\x{gl.dim}(A)=\dim P(H)+1$, and thus Lemma \ref{20} forces\begin{align*}\dim P(H)<\x{gl.dim}(A)\leq\x{gl.dim}(H),\end{align*}which is a contradiction to our assumption. Therefore $H=U_H$.
\par2. If we assume that $U_H$ is normal, then consider any non-primitive $t\in P_2(H)$. By Lemma \ref{15}, $[t,P(H)]\subset P(H)$. Let $A$ be the Hopf subalgebra of $H$ generated by $P(H)\oplus kt$. It follows that $A$ is primitively thick containing $U_H$, thus it must be algebra-isomorphic to an enveloping algebra $U(L)$, where $L$ is a Lie algebra of dimension $\dim P(H)+1$. Hence $A$ is a Noetherian Hopf subalgebra with $\x{gl.dim}(A)=\dim P(H)+1$. Applying Lemma \ref{20} yields\begin{align*}\dim P(H)<\x{gl.dim}(A)\leq\x{gl.dim}(H),\end{align*}a contradiction. Therefore $H=U_H$.\end{proof}

\par Theorem \ref{02} also provides us a strict lower bound of Corollary \ref{21}, namely $\dim P(H)<\x{gl.dim}(H)$, whenever condition 1 or 2 hold. Relating to GK-dimension, a consequence of Theorem \ref{02} shows that $\x{gl.dim}(H)>\dim P(H)$ if and only if $\x{GK.dim}(H)>\dim P(H)$.
\par Futhermore, we see that Theorem \ref{02} is a specific setting for a more general question.

\begin{que}If $A\subset H$ are (affine) locally finite connected Hopf algebras with $\x{gl.dim}(H)=\x{gl.dim}(A)$, does $H=A$?\end{que}

\section*{Acknowledgments}
The author would like to thank his adviser Allen Bell, and the Mathematics Department at both University of Wisconsin-Milwaukee and Bradley University for their guidance and support of the author's work. The author is also grateful to the referees for their constructive feedback and advice.

\vspace{.2in}

\par\scshape{Department of Mathematics, Bradley University, Peoria, Illinois 61606}
\par\textit{E-mail addresss}: \texttt{dyee@fsmail.bradley.edu}
\end{document}